\documentclass[11pt,reqno]{amsart}

\usepackage{enumerate, amsmath, amsthm, amsfonts, amssymb, 
mathrsfs, graphicx}
\usepackage[usenames, dvipsnames]{color}
\usepackage[margin=1in]{geometry} 
\usepackage{hyperref}

\numberwithin{equation}{section}
\newtheorem{Theorem}[equation]{Theorem}
\newtheorem{Proposition}[equation]{Proposition}
\newtheorem{Lemma}[equation]{Lemma}

\newtheorem{Conjecture}[equation]{Conjecture}

\newtheorem{Problem}[equation]{Problem}

\theoremstyle{definition}

\newtheorem{eg}[equation]{Example}

\newcommand{\fg}{\mathfrak{g}}

\newcommand{\fm}{\mathfrak{m}}

\newcommand{\NN}{\mathbb{N}}

\newcommand{\ZZ}{\mathbb{Z}}

\renewcommand{\phi}{\varphi}

\definecolor{red}{rgb}{1,0,0}
\definecolor{orange}{rgb}{1,0.5,0}
\definecolor{purple}{rgb}{.5,.2,.8}
\definecolor{blue}{rgb}{.2,.2,.8}
\definecolor{green}{rgb}{.4,.6,.4}

\DeclareMathOperator{\supp}{Supp}

\DeclareMathOperator{\mult}{mult}

\newcommand{\twomat}[4]{\begin{bmatrix}
        #1 & #2  \\[1em]
        #3 & #4 \\
     \end{bmatrix}}
\newcommand{\threemat}[9]{\begin{bmatrix}
        #1 & #2 & #3 \\[1em]
        #4 & #5 & #6 \\[1em]
        #7 & #8 & #9 
     \end{bmatrix}}
\newcommand{\suchthat}{\mid}

\newcommand{\real}{\mathrm{re}}
\newcommand{\fin}{\mathrm{fin}}
\newcommand{\height}{\mathrm{ht}}
\newcommand{\imag}{\mathrm{im}} 

\newcommand{\m}{\fm}

\newcommand{\s}{s}
\renewcommand{\sigma}{s}

\begin{document}

\title[Correction factors for Kac-Moody groups]{Correction factors for Kac-Moody groups and $t$-deformed root multiplicities}

\author{Dinakar Muthiah}
\address{Department of Mathematics and Statistics, University of Massachusetts Amherst, Amherst, MA 01003}
\email{muthiah@math.umass.edu}

\author{Anna Pusk{\'{a}}s}
\address{Department of Mathematics and Statistics, University of Massachusetts Amherst, Amherst, MA 01003}
\email{puskas@math.umass.edu}

\author{Ian Whitehead}
\address{Department of Mathematics, Statistics, and Computer Science, Macalester College, St. Paul, MN 55105}
\email{iwhitehe@macalester.edu}
\maketitle

\begin{abstract}
We study a correction factor for Kac-Moody root systems which arises in the theory of $p$-adic Kac-Moody groups. In affine type, this factor is known, and its explicit computation is the content of the Macdonald constant term conjecture. The data of the correction factor can be encoded as a collection of polynomials $m_\lambda \in \ZZ[t]$ indexed by positive imaginary roots $\lambda$. At $t=0$ these polynomials evaluate to the root multiplicities, so we consider $m_\lambda$ to be a $t$-deformation of $\mult (\lambda)$. We generalize the Peterson algorithm and the Berman-Moody formula for root multiplicities to compute $m_\lambda$. As a consequence we deduce fundamental properties of $m_\lambda$. 
\end{abstract}
\section{Introduction}\label{sect:intro}
In \cite{Macdonald72}, Macdonald proved the following remarkable identity for finite root systems $\Phi$:
\begin{align}
  \label{eq:15}
\sum_{w \in W} t^{\ell(w)} = \sum_{w \in W} w \left( \prod_{\alpha \in \Phi^+} \frac{1 - t e ^\alpha}{1 - e^\alpha} \right)
\end{align}
Here $W$ is the Weyl group, $\ell:W \rightarrow \ZZ_{\geq 0}$ is the length function, and $\Phi^+$ is the set of positive roots. Still more remarkably, this formula no longer holds for infinite type Kac-Moody root systems. Instead one has:
\begin{align}
  \label{eq:16}
\sum_{w \in W} t^{\ell(w)} = \m \sum_{w \in W} w \left(\prod_{\alpha \in \Phi_{\real}^+} \frac{1 - t e ^\alpha}{1 - e^\alpha}\right)
\end{align}
where $\Phi^+_{\real}$ is now the set of positive {\em{real}} roots, and the correction factor $\m$ is a $W$-invariant power series supported on the positive imaginary cone of the root lattice. In \cite{Macdonald03}, Macdonald computes $\m$ in the affine Kac-Moody case. An explicit formula for $\m$ can be obtained from the so-called constant term of the Macdonald kernel. In fact, the computation of $\m$ in affine type is equivalent to the Macdonald constant term conjecture \cite{Macdonald82}, proven by Cherednik \cite{Cherednik95}. We will review the formula for $\m$ in affine type in \S \ref{sect:affinecase}.

In this paper, our goal is to initiate the explicit study of the correction factor $\m$ beyond affine type to arbitrary Kac-Moody type. In \S \ref{sect:peterson} we give an algorithm to compute the power series $\m$, based only on knowledge of the real roots in $\Phi$. In \S \ref{sect:GenBermanMoody} we give a closed formula for $\m$. These tools are used to establish properties of $\m$ which hold across all finite, affine, and indefinite Kac-Moody types. 

In addition to its appearance in \eqref{eq:16} the power series $\m$ appears as a correction factor in many other formulas. In this sense $\m$ captures much of the complexity that arises when passing from finite to infinite type. Below, we briefly recall some situations where $\m$ appears. 

\subsubsection{$p$-adic groups and Macdonald's formula for the spherical function}

Our initial motivation for studying the factor $\m$ comes from the theory of $p$-adic Kac-Moody groups. In \cite{BGKP}, Braverman, Garland, Kazhdan, and Patnaik show that the discrepancy between the affine Gindikin-Karpelevich formula and the naive generalization from finite type is this factor $\m$. The same is true in Macdonald's formula for the spherical function \cite{BKP} and the Casselman-Shalika formula for the spherical Whittaker function \cite{Patnaik}. Work of Patnaik and the second author \cite{patnaik2017metaplectic} extends these formulas to metaplectic $p$-adic Kac-Moody groups, and $\m$ similarly appears. 

Gaussent and Rousseau have recently initiated the study of (not necessarily affine) $p$-adic Kac-Moody groups using their theory of {\em {masures}}, previously known as {\em{hovels}}. The work of  Bardy-Panse, Gaussent, Rousseau \cite[Theorem 7.3]{BGR} show that the role of $\m$ in affine type generalizes to arbitrary Kac-Moody type for the Macdonald formula for the spherical function. We expect that this should continue to hold for the Gindikin-Karpelevich and Casselman-Shalika formulas. We mention H\'ebert's \cite{Hebert-17} work on the well-definedness of Gindikin-Karpelevich formula in the general Kac-Moody setting.

The general principle appears to be that formulas from finite dimensional $p$-adic groups directly generalize to the Kac-Moody case, and the sole new phenomenon is the appearance of the factor $\m$.

\subsubsection{Symmetrizers for Kac-Moody Hecke algebras}
Identities relating Hecke symmetrizers, the Weyl-Kac characters formula, and the factor $\m$  appear in the study of the infinite-dimensional (ordinary and metaplectic) generalizations of the Casselman-Shalika formula for spherical Whittaker functions. The classical Casselman-Shalika formula expresses a spherical Whittaker function in terms of the character of a representation. By the Weyl character formula, this character can be produced by acting on a highest-weight monomial by a ``symmetrizer'', which is a certain sum over the Weyl group. 

A Hecke symmetrizer is a sum: 
\begin{equation}\label{eq:Hecke_symm}
\sum_{w\in W} T_w
\end{equation}
where the $T_w$ are Demazure-Lusztig operators. Each $T_w$ can be expanded in terms of Weyl group elements that are bounded by $w$ in the Bruhat order, giving rise to an identity relating Weyl and Hecke symmetrizers.

The spherical Whittaker function can be expressed in terms of an operator as in \eqref{eq:Hecke_symm} acting on a highest weight monomial in finite, affine, and indefinite Kac-Moody type. In the finite dimensional setting the sum \eqref{eq:Hecke_symm} satisfies an {\em{operator Casselman-Shalika formula}} \cite[\S 2.12]{Patnaik}, an identity of operators similar to the Demazure character formula, connecting the Whittaker function expressed with \eqref{eq:Hecke_symm} to its usual expression in terms of a Weyl character. In this case the correction factor is just $\m=1.$ 

The analogous identity between \eqref{eq:Hecke_symm} and the Weyl-Kac character in the affine and indefinite Kac-Moody setting \cite{Patnaik,patnaik2017metaplectic} involves the Weyl invariant correction factor $\m .$ In this case, both symmetrizers are infinite sums. Indeed, the proof that the Hecke symmetrizer \eqref{eq:Hecke_symm} yields a well-defined operator depends on its relation to the Weyl symmetrizer. The proof proceeds by comparing two similar versions of \eqref{eq:Hecke_symm} corresponding to spherical and Whittaker functions respectively. (See \cite[Section 3]{patnaik2017metaplectic} for details.)

\subsubsection{Dolbeault cohomology and the failure of the Hodge theorem}

As Macdonald explains in \cite{Macdonald72}, formula \eqref{eq:15} can be interpreted as a computation of the Betti numbers of a flag variety using Hodge theory. The left hand side is the computation given by counting Schubert cells. The right hand side corresponds to a computation of Dolbeault cohomology using localization at fixed points for the action of the maximal torus. Because the flag variety is smooth and projective, the Hodge theorem establishes that Dolbeault cohomology is equal to Betti cohomology. Equation \eqref{eq:15} results.

The fact that this argument no longer holds in Kac-Moody type is an indication that the Hodge theorem no longer holds for Kac-Moody flag varieties. This is very counterintuitive. It is known that Kac-Moody flag varieties are (ind-)projective. This means that they are not smooth for the purposes of the Hodge theorem. Because they are homogeneous for a group action, Kac-Moody flag varieties are in a sense everywhere singular.

In affine type, the situation is well-understood by the work of Fishel-Grojnowski-Teleman \cite{Fishel-Grojnowski-Teleman08}. They explicitly compute the Dolbeault cohomology of the affine flag variety. In particular, their work shows how the Hodge theorem fails in affine type. One of their results is a proof of the ``strong Macdonald conjecture,'' a conjecture on Lie algebra cohomology which implies the constant term conjecture.

The work of \cite{Fishel-Grojnowski-Teleman08} suggests that the generalization of the strong Macdonald conjecture to general Kac-Moody types should consist of an explicit computation of Dolbeault cohomology for Kac-Moody flag varieties. This should contain strictly more information than the factor $\m$ because one loses the bigrading of Dolbeault cohomology. In finite type only $(p,p)$ classes occur, so no information is lost, but in general other classes must occur. 

\subsection{Our work}

Our approach is based on the idea that the problem of computing $\m$ generalizes the problem of computing multiplicities of imaginary roots. In fact, when specialized at $t=0$, a product formula for $\m$ exactly encodes the imaginary root multiplicities. Our main results are two methods for computing this product formula. The methods are $t$-deformations of two existing approaches to computing imaginary root multiplicities: the Peterson algorithm \cite{Peterson83}, \cite{kang2001peterson}, \cite[Exercises 11.11, 11.12]{Kac90} and the Berman-Moody formula \cite{Berman-Moody79}.

Let us write $Q^+_{\imag}$ for the {\it positive imaginary root cone}, which consists of all elements of $Q_+$ (the positive root cone) which remain in $Q_+$ after applying any element of the Weyl group. Our generalized Peterson algorithm is a recursive procedure that allows us to express:
\begin{align}
  \label{eq:2}
\m=\prod_{\lambda \in Q^+_{\imag}} \prod_{n\geq0}(1-t^n e^{\lambda})^{-m(\lambda, n)}
\end{align}
For fixed $\lambda$, the $m(\lambda,n)$ are integers that vanish for $n$ sufficiently large (the bound depends on $\lambda$). The exponent $m(\lambda,0)$ is the multiplicity of $\lambda$. In particular it lies in $\ZZ_{\geq 0}$ and it is non-zero precisely when $\lambda$ is a root.  Our generalized Peterson algorithm is fast, depending only on operations with polynomials. We have implemented it in SageMath, and a small sample of the output is included in the appendix.

It is convenient to package the coefficients $m(\lambda,n)$ into a polynomial $\displaystyle m_\lambda(t) = \sum_{n\geq0} m(\lambda,n) t^n$. The generalized Berman-Moody formula is the following closed expression for the polynomial $m_\lambda(t)$.

\newtheorem*{thm:BermanMoody}{Theorem \ref{thm:BermanMoody}}
\begin{thm:BermanMoody}
For all $\lambda \in Q_{\imag}^+$, we have: 
\begin{equation} \label{eq:multformula-intro}
m_\lambda(t) = \sum_{\kappa | \lambda} \mu\left(\lambda/\kappa\right)\left(\frac{\lambda}{\kappa}\right)^{-1} \sum_{\underline{\kappa}\in \text{Par}(\kappa)} (-1)^{|\underline{\kappa}|}\frac{B(\underline{\kappa})}{|\underline{\kappa}|} \prod_{i=1}^{|\underline{\kappa}|} P_{\kappa_i}(t^{\lambda/\kappa})
\end{equation}
\end{thm:BermanMoody}

We refer the reader to \S \ref{sect:GenBermanMoody} for an explanation of the symbols appearing here. Although not as efficient as the generalized Peterson algorithm, this formula is useful for answering theoretical questions about $\m$. For example, we prove the following:

\newtheorem*{thm:roots_only}{Theorem \ref{thm:roots_only}}
\begin{thm:roots_only}
For all $\lambda \in Q^+_{\imag}$, $m_\lambda(t)$ is nonzero if and only if $\lambda$ is a root. 
\end{thm:roots_only}

\newtheorem*{thm:order2vanishing}{Theorem \ref{thm:order2vanishing}}
\begin{thm:order2vanishing} 
For imaginary roots $\lambda$, the polynomial $m_\lambda(t)$ is divisible by $(1-t)^2$.
\end{thm:order2vanishing}

We note that in affine type Theorem \ref{thm:order2vanishing} is inherent in Macdonald's formatting of the constant term conjecture. Moreover, in affine type, the polynomial $\chi_\lambda(t) = \frac{m_\lambda(t)}{(1-t)^2}$ is more fundamental. It has positive coefficients which are manifestly equal to the cardinalities of certain finite sets. For example, in untwisted affine type, if $\lambda$ is imaginary then $\chi_\lambda(t)$ is the generating polynomial of heights of positive roots of the underlying finite root system (see \S \ref{sect:affinecase} for details).

Beyond affine type the polynomials $\chi_\lambda$ no longer have positive coefficients. However, from computer calculation via the generalized Peterson algorithm, we conjecture the following.
\begin{Conjecture} \label{conj:alternating}
In rank two hyperbolic type, the polynomials $\chi_\lambda$ have alternating sign coefficients. 
\end{Conjecture}
Beyond rank two, this does not hold. It is conceivable that the positivity phenomena of affine type is present in the general $\chi_\lambda$ in a more subtle form.

\subsection{Further directions}

The constant term of the polynomial $\chi_\lambda$ is precisely the multiplicity of the root $\lambda.$ It must be positive as it is the dimension of the corresponding root space in the Kac-Moody Lie algebra. It would be illuminating to give a similar interpretation for the other coefficients.

\begin{Problem}
Interpret all coefficients of $\chi_\lambda$ in terms of the Kac-Moody Lie algebra.
\end{Problem}

Computing imaginary root multiplicities is a well-known difficult problem. We propose that one should try to instead study $\chi_\lambda$, which of course contains the information of root multiplicities. Our hope is that this generalized problem might be more tractable. For example, a long-standing problem is Frenkel's conjecture, which provides a conjectural bound for hyperbolic root multiplicities. A natural question is to phrase (and prove) a generalization of this conjecture.

\begin{Problem}
Give upper bounds for the degree and coefficients of $\chi_\lambda(t)$ in terms of $\lambda$.
\end{Problem}

\subsection{Acknowledgments}

We thank Alexander Braverman, Paul E. Gunnells, Kyu-Hwan Lee, Dongwen Liu, Peter McNamara, Manish Patnaik for helpful conversations. At the beginning of this project the first author was partially supported by a PIMS postdoctoral fellowship and the second author was supported through Manish Patnaik's Subbarao Professorship in Number Theory and an NSERC Discovery Grant at the University of Alberta. The project started at the workshop ``Whittaker functions: Number Theory, Geometry and Physics'' at the Banff International Research Station in 2016; we thank the organizers of this workshop. 


\section{Preliminaries and Notation}\label{sect:prelim}

\subsection{The inverse with respect to a cone}
We begin by introducing some terminology about lattices, cones, and series supported on subsets of them. Let $Q$ be a lattice, and let $Q^+ \subseteq Q$ be a subset. We say $Q^+$ is a \emph{cone} if it is closed under taking linear combinations with non-negative integer coefficients. In particular a cone contains $0.$ For every $\gamma \in Q^+$ one may consider the set of pairs $(\alpha,\beta) \in (Q^+)^2$ such that $\gamma = \alpha + \beta$. We call $Q^+$ a \emph{strict cone} if the set of such pairs is finite for every element. We say that $Q^+$ is \emph{graded} if there is an additive height function $\text{ht}:Q^+ \to \mathbb{Z}_{\geq 0}$ such that 
the set of vectors in $Q^+$ of a fixed height is finite. Note that a graded cone is strict. We call $Q^+$ \emph{saturated} if $Q^+ \otimes \mathbb{R}_{\geq0} \cap Q = Q^+$. Finally, we say that a saturated cone $Q^+$ has \emph{full rank} if $Q^+$ generates the lattice $Q$. In this case $Q=\{\alpha-\beta\mid \alpha, \beta\in Q^+\}.$

Consider a formal sum of exponentials $f=\sum_{\lambda\in Q} c_{\lambda}e^{\lambda}$ with coefficients $c_{\lambda}$ in a ring $R$. We say the sum $f$ has support $\supp(f)=\lbrace \lambda \in Q\mid  c_{\lambda} \neq 0 \rbrace$. If $\supp(f)\subseteq Q^+$ we call $f$ a power series with respect to $Q^+$ and write $f\in R[[Q^+]]$. If the support of $f$ is contained in the translate of $Q^+$ by some $\lambda_0 \in Q$, then we call $f$ a Laurent series with respect to $Q^+$. If $Q^{+}$ is a strict cone then the power series and Laurent series with respect to $Q^+$ form a ring. If $Q^{+}$ has full rank then if $\supp(f)$ is finite then $f$ is a Laurent series. 

The graded cone condition allows us to characterize units and inverses in $ R[[Q^+]].$ 
A power series of the form $f=\sum_{\lambda \in Q^+} c_{\lambda} e^{\lambda}$ is a unit if and only if $c_0 \in R$ is a unit. The inverse $f^{-1}$ can be constructed explicitly via induction on height. We will use a more general version of this. Suppose that $Q^+$ is a graded cone and $Q^{++}$ is a subcone. We denote the restriction of a power series $f=\sum_{\lambda \in Q^+} c_{\lambda} e^{\lambda}$ to $Q^{++}$ as $f|_{Q^{++}}=\sum_{\lambda\in Q^{++}} c_{\lambda} e^{\lambda}$. We have the following:

\begin{Proposition} \label{prop:relativeinverse} Suppose that $Q^+$ is a graded cone, and $Q^{++}$ is a subcone.  If $f\in R[[Q^+]]$ is a unit, then there exists a unique 
 $g\in R[[Q^{++}]]$ such that such that $(fg)|_{Q^{++}}=1$. Furthermore, $g$ is a unit in $R[[Q^{++}]]$. \end{Proposition}

The series $g$ may be constructed explicitly by induction on height. It is manifest from the construction that $g$ is unique. We call $g$ {\it the inverse of $f$ relative to $Q^{++}$}. The ring of Laurent series has analogous constructions; there a unit has the form $f=\sum_{\lambda \in Q^+} c_{\lambda} e^{\lambda+\lambda_0}$, with $\lambda_0\in Q$ and $c_0$ a unit in $R$.

\subsection{Weyl action and the positive imaginary cone}

Let $\Phi$ be a Kac-Moody root system. We refer the reader to \cite{Kac90}, especially Chapter 5, for details on such root systems. Let $\alpha_1, \ldots \alpha_r$ denote the simple roots of $\Phi,$ $\sigma_1, \ldots \sigma_r$ the corresponding simple reflections and let $W$ be the Weyl group. Let $Q$ be the root lattice, and $Q^+$ the set of positive integer linear combinations of simple roots. Then $Q^+$ is a saturated graded cone of full rank in $Q$. We shall consider rings of power series and Laurent series with coefficients in several different rings: $\ZZ$, $\ZZ[t]$, $\ZZ[t, t^{-1}]$ and $\ZZ[[t]][t^{-1}]$.

A Weyl group element $w\in W$ acts on a formal sum $f=\sum_{\lambda} c_{\lambda}e^{\lambda}$ by $w(f)=\sum_{\lambda} c_{\lambda}e^{w(\lambda)}$. This action takes a power series (resp. a Laurent series) with respect to $Q^+$ to a power series (resp. Laurent series) with respect to $w(Q^+)$. A certain subring of power series (resp. Laurent series) with respect to  $Q^+$ are also power series (resp. Laurent series) with respect to $w(Q^+)$ for all $w \in W.$ A priori, we only have a well-defined $W$-action on this subring.

We define the positive imaginary cone $Q^+_{\imag}$ as the cone generated by positive imaginary roots, and call $\lambda \in Q^+$ imaginary if it lies in this cone. 
We can make a similar definition for $Q^-_{\imag}$. The imaginary vectors form $0$ and $1$ dimensional subspaces in finite and affine type. In indefinite type by \cite[Proposition 5.8, Theorem 5.4]{Kac90} the positive imaginary roots generate a saturated graded cone of full rank. This cone is $Q^+_{\imag}$. It is $W$-invariant; indeed it is the largest $W$-invariant subcone of $Q^+$ \cite[Theorem 5.4]{Kac90}. That is, every non-imaginary vector $\lambda$ is $W$-equivalent to one outside $Q^+$. By contrast, each imaginary $\lambda$ is $W$-equivalent to a vector in $\lambda' \in Q^+$ of minimal height, which is necessarily antidominant, i.e.\ $\langle \lambda, \alpha_i \rangle \leq 0$ for all simple roots $\alpha_i$. If $\lambda \in Q^+$ satisfies $\langle \lambda, \lambda \rangle >0$, then $\langle \lambda, \alpha_i \rangle >0$ for some simple root $\alpha_i$, and since $\langle \lambda, \lambda \rangle$ is $W$-invariant, we see then that $\lambda$ is imaginary only if $\langle \lambda, \lambda \rangle \leq 0$. This condition is sufficient in finite, affine, and hyperbolic type \cite[Proposition 5.10]{Kac90} but not in general.

We define the imaginary part of a formal sum $f=\sum_{\lambda} c_{\lambda}e^{\lambda}$ as $\text{Im}(f)=f|_{Q^+_{\imag}}$, and call $f$ imaginary if $\supp(f)\subseteq Q^+_{\imag}$. In light of the discussion above, we have $Q^+_{\imag}=\bigcap_{w\in W}w(Q^+)$, so $R[[Q^+_{\imag}]]$ is the largest subring of $R[[Q^+]]$ with a $W$-action. 
The situation with Laurent series is somewhat more subtle. A formal sum which is a Laurent series with respect to $w(Q^+)$ for each $w\in W$ may not be a Laurent series with respect to $Q^+_{\imag}$. This is clear in finite and affine type and it persists in indefinite type. For a counterexample consider a series $\sum_{\alpha} e^{n_{\alpha} \alpha}$, where the $\alpha$ are positive real roots, and the vectors $n_{\alpha}\alpha$ grow arbitrarily far from the imaginary cone. 

\subsection{Good products}

We wish to extend the $W$-action to a larger class of Laurent series units. We shall make use of the following proposition, which allows us to represent a Laurent series unit as a product of linear terms. 

\begin{Proposition}\label{prop:seriestoproduct} Suppose that $f$ is a Laurent series unit with respect to a graded cone $Q^+$, with coefficients in $\ZZ[[t]][t^{-1}]$. Then $f$ has a unique product form:
\begin{equation} \label{eq:productform}
f=ue^{\lambda_0}\prod_{\lambda \in Q^+\setminus \lbrace0 \rbrace} \prod_{n} (1-t^n e^{\lambda})^{m(\lambda, n)}
\end{equation}
where $m(\lambda, n)\in \ZZ$ and the set $\{n\in \ZZ\mid m(\lambda, n)\neq 0\}$ is bounded from below for each $\lambda$. If furthermore $f$ has coefficients in $\ZZ[t, t^{-1}]$, $\ZZ[t],$ or $\ZZ$ then this set is a bounded subset of $\ZZ,$ $\ZZ_{\geq 0},$ or $\{0\},$ respectively. 
\end{Proposition}

%

The strict cone condition implies that the Laurent series expansion of an expression of type \eqref{eq:productform} is well-defined. It can be shown by induction on $\height(\lambda)$ that every Laurent series is the expansion of some product of the form \eqref{eq:productform}, and that two distinct products of this form yield distinct series.

We call a product of the form \eqref{eq:productform} a {\emph{good product}} if all vectors $\lambda$ in the product are multiples of roots $\alpha\in \Phi$, and the set of factors corresponding to any real root $\alpha$ is finite. Good products form a multiplicative subgroup of all Laurent series units. Next we define a Weyl group action on this set. For $\lambda \in \mathbb{N} \Phi$, let:
\begin{equation}
\label{eq:4}
w(1 - t^n e^\lambda) = 
\left\lbrace \begin{array}{cc}
1 - t^n e^{w(\lambda)} & \text{if } w(\lambda) > 0 \\
(- t^n e^{w(\lambda)}) (1 - t^{-n} e^{-w(\lambda)}) & \text{if } w(\lambda) < 0
\end{array} \right.
\end{equation}  
This gives an action of $W$ on linear factors which extends to all good products by multiplicativity. The restriction to multiples of roots is necessary so that $w(\lambda)$ is either in $Q^+$ or in $Q^-$. The finiteness condition guarantees that the $w$-image of a good product is a well-defined Laurent series, since the inversion set $\Phi(w)=\Phi^+ \cap w^{-1}(\Phi^-)$ is finite. Earlier, we defined a $W$-action on formal sums of exponentials. That action respects Laurent series multiplication. It follows that the two $W$-actions agree when both make sense. However, the action on good products is well-defined in some cases where the action on Laurent series is not. For the simplest example consider  $\sigma_i((1-e^{\alpha_i})^{-1})$. It is the good product action which allows us to interpret the definition of $\m$ in \eqref{eq:16}. 


\section{A Formal Definition of the Factor $\m$}\label{sect:thefactorm}
\renewcommand{\P}{\Omega}

\subsection{The definition of $\m$}
Let $\P_0$, $\P_1$, and $\P_2$ denote the power series rings with respect to the positive root lattice $Q^+$, with coefficients in $\ZZ$, $\ZZ[t, t^{-1}]$, and $\ZZ[[t]][t^{-1}]$ respectively. 

We wish to use \eqref{eq:16} as the definition of $\m$. However, it is not {\it a priori} clear what type of object $\m$ is. Below, we will define $\m$ as an element of $\P_2$ and later show that it lies in $\P_1$.

Let:
\begin{equation}
\Delta_{\real}=\prod_{\alpha \in \Phi^+_{\real}}(1-e^{\alpha}), \qquad \Delta_{t, \real}=\prod_{\alpha \in \Phi^+_{\real}}(1-t e^{\alpha})
\end{equation}
power series in $\P_0$ and $\P_1$ respectively. The quotient $\frac{\Delta_{t, \real}}{\Delta_{\real}}$ is a unit and a good product in $\P_1$. Let $P(t)\in \ZZ[[t]]$ be the Poincar\'{e} series $P(t)=\sum_{w \in W} t^{\ell(w)}$, where $\ell(w)$ denotes the length of the Weyl group element $w$ in the simple reflections. 

\begin{Lemma}
  Consider the sum
  \begin{align}
    \label{eq:3}
    \sum_{w\in W} w\left( \frac{\Delta_{t, \real}}{\Delta_{\real}}\right) 
  \end{align}
  interpreting each term by the good product action of $W$. It converges to a unit in $\P_2$ that is regular at $t=0$. Its constant coefficient is $P(t)$.  
\end{Lemma}

\begin{proof}
  
According to the definition of the good product action: 
\begin{align}
w\left( \frac{\Delta_{t, \real}}{\Delta_{\real}}\right) &= \prod_{\alpha \in \Phi(w)} t\left(\frac{1-t^{-1}e^{-w(\alpha)}}{1-e^{-w(\alpha)}}\right) \prod_{\alpha\in \Phi_{\real}-\Phi(w)} \left(\frac{1-t e^{w(\alpha)}}{1-e^{w(\alpha)}}\right) \nonumber \\
&=t^{\ell(w)} \prod_{\alpha \in \Phi(w^{-1})} \left(\frac{1-t^{-1}e^{\alpha}}{1-e^{\alpha}}\right)\prod_{\alpha\in \Phi_{\real}-\Phi(w^{-1})} \left(\frac{1-t e^{\alpha}}{1-e^{\alpha}}\right)
\end{align}
where $\Phi(w)=\Phi^+ \cap w^{-1}(\Phi^-)$ and $\Phi(w^{-1})=\Phi^+ \cap w(\Phi^-)$ are inversion sets, both of size $\ell(w)$. Notice that this expression remains a power series unit in $\P_1$, 
because of the structure of the product $\frac{\Delta_{t, \real}}{\Delta_{\real}}$. Moreover, this expression is regular at $t=0$, and the coefficient of some $\lambda\in Q^+$ is divisible by at least $t^{\ell(w)-\height(\lambda)}$. Since there are finitely many elements of $W$ for any fixed length, it follows that the coefficient of $\lambda$ in the sum over all $w\in W$ is a well-defined power series in $\ZZ[[t]]$. Therefore \eqref{eq:3} is a well-defined power series in $\P_2$, regular at $t=0$. The constant coefficient of each $w$-summand is $t^{\ell(w)}$, hence the total constant coefficient is precisely $P(t)$.
\end{proof}

As $P(t)$ is also a unit in $\P_2$ regular at $t=0$, we define $\m$ by the expression:
\begin{equation} \label{eq:mdefinition}
\m \sum_{w\in W} w\left( \frac{\Delta_{t, \real}}{\Delta_{\real}}\right) = P(t)
\end{equation}
We conclude that $\m$ is a unit in $\P_2$, regular at $t=0$, with constant coefficient $1$.

\subsection{Weyl invariance of $\m$}
Next we show that $\m$ is invariant under $W$, and is therefore supported on $Q^+_{\imag}$, the maximal the maximal Weyl invariant subcone of $Q^+$. In particular this implies that $\m$ is $1$ when $\Phi$ is of finite type, and is a single-variable power series in affine type. Weyl invariance is strongly suggested by the form of \eqref{eq:mdefinition}, but the proof requires care. We record the following useful identities:
\begin{align}
&\frac{\Delta_{\real}}{w(\Delta_{\real})}=(-1)^{\ell(w)} \prod_{\alpha \in \Phi(w^{-1})} e^{\alpha} \\
&\frac{w(\Delta_{t, \real})}{\Delta_{t, \real}}= \prod_{\alpha \in \Phi(w^{-1})}-te^{-\alpha}\left(\frac{1-t^{-1} e^{\alpha}}{1-t e^{\alpha}} \right)
\end{align}
which are both interpreted via the good product action of $W$. In each case the left side is a ratio of infinite products, but all but the finitely many factors corresponding to $\Phi(w^{-1})$ cancel out.

Multiply both sides of \eqref{eq:mdefinition} by $\Delta_{\real}$ to obtain:
\begin{equation} \label{eq:minvariant}
\m \sum_{w\in W} t^{\ell(w)} \prod_{\alpha \in \Phi(w^{-1})} (1-t^{-1}e^{\alpha}) \prod_{\alpha\in \Phi_{\real}-\Phi(w^{-1})} (1-t e^{\alpha})=\Delta_{\real} P(t)
\end{equation}
Consider both sides of this equation as power series in $\P_2$ (not as good products) and act on them by a simple reflection $\sigma_i$. The right side becomes:
\begin{equation}
\sigma_i(\Delta_{\real}) P(t) =-e^{-\alpha_i} \Delta_{\real} P(t)
\end{equation}
One may check that applying $\sigma_i$ to the sum over $W$ on the left side interchanges the summands of $w$ and $\sigma_i w$, and multiplies both by $-e^{-\alpha_i}$. 

Comparing the original form of \eqref{eq:minvariant} to the equality obtained after applying $\sigma_i$, we see immediately that $\sigma_i(\m)=\m$. Thus $\m$ is fixed by $W$, and therefore is a power series with respect to the imaginary subcone $Q^+_{\imag}$. 

\subsection{Viewing $\m^{-1}$ as an inverse with respect to the positive imaginary cone}
 
Our final goal in this section is to characterize $\m^{-1}$ as the inverse of $\frac{\Delta_{\real}}{\Delta_{t, \real}}$ relative to the cone $Q^+_{\imag}$, in the sense of Proposition \ref{prop:relativeinverse}. Recall that though $\m$ and $\m^{-1}$ were originally defined in the ring $\P_2$ since $\frac{\Delta_{\real}}{\Delta_{t, \real}} \in \P_1$ it follows that $\m^{-1}$ and $\m$ in fact belong to $\P_1$ as well.

Since $\m$ is a power series unit on the cone $Q^+_{\imag}$, its inverse $\m^{-1}$ is one as well. Multiplying \eqref{eq:mdefinition} by $\frac{\Delta_{\real}}{\m \Delta_{t, \real}}$ we obtain:
\begin{equation} \label{eq:identity}
\sum_{w\in W} t^{\ell(w)} \prod_{\alpha \in \Phi(w^{-1})} \frac{1-t^{-1}e^{\alpha}}{1-te^{\alpha}} = \frac{P(t) \Delta_{\real}}{\m \Delta_{t, \real}}
\end{equation}
Both sides are power series units in $\P_2$, regular at $t=0$, with constant coefficient $P(t)$. Moreover, observe that the restriction of the left side to the cone $Q^+_{\imag}$ is only $P(t)$. Indeed, each $w$ summand is a power series supported in the cone $Q^+ \cap w(Q^-)$, which intersects the imaginary cone only at the origin. Therefore the support of the left side does not contain any nonzero imaginary $\lambda .$ We conclude that:
\begin{equation} \label{eq:imaginarypart}
\left. \left(\m^{-1}\frac{\Delta_{\real}}{\Delta_{t, \real}}\right) \right|_{Q^+_{\imag}}=1
\end{equation}
By Proposition \ref{prop:relativeinverse}, this equation characterizes $\m^{-1}$ uniquely. 

We summarize the results of this section in the following theorem:

\begin{Theorem}
The series $\m$ and $\m^{-1}$ are power series units in $\ZZ[t,t^{-1}][[Q^{+}]]$, regular at $t=0$, with constant coefficient $1$. Both are fixed by the Weyl group $W$ and supported on the positive imaginary cone $Q^+_{\imag}$, and $\m^{-1}$ is the inverse of $\frac{\Delta_{\real}}{\Delta_{t, \real}}$ relative to $Q^+_{\imag}$.
\end{Theorem}

\subsection{Remarks}
We conclude this section with two remarks: first, in affine and indefinite Kac-Moody type, the full Weyl denominator $\Delta$ and $t$-twisted Weyl denominator $\Delta_t$ are defined similarly to $\Delta_{\real}$ and $\Delta_{t, \real}$, but with additional factors of $(1-e^{\alpha})^{\text{mult}(\alpha)}$ and $(1-te^{\alpha})^{\text{mult}(\alpha)}$, respectively, for all positive imaginary roots $\alpha$. The positive integers $\text{mult}(\alpha)$, called root multiplicities, have been the subject of extensive study. We have defined $\m$ to include the root multiplicities as part and parcel. Indeed, specializing equation \eqref{eq:identity} at $t=0$ gives the Kac-Weyl denominator formula:
\begin{equation} \label{eq:denominatorformula}
\sum_{w\in W} (-1)^{\ell(w)} \prod_{\alpha \in \Phi(w^{-1})} e^{\alpha} = \Delta_{\real} \m^{-1}|_{t=0}
\end{equation}
from which we see that the specialization of $\m^{-1}$ at $t=0$ is $\prod_{\alpha \in \Phi^+_{\imag}} (1-e^{\alpha})^{\text{mult}(\alpha)}$. Specializing equation \eqref{eq:imaginarypart} at $t=0$ shows that this product is the inverse of $\Delta_{\real}$ relative to the cone $Q^+_{\imag}$. The Peterson algorithm \cite{Peterson83} and the formula of Berman and Moody \cite{Berman-Moody79} are means of computing imaginary roots and their multiplicities based on knowledge of the real roots in $\Phi$. We generalize these results to compute all of $\m$ in the sequel. 

The second remark concerns root systems $\Phi$ of affine type. In this case, $Q^+_{\imag}$ is a one-dimensional cone generated by the minimal imaginary root $\delta$. This means that $\m$ is a single-variable power series. The restriction map $|_{Q^+_{\imag}}$ is linear over imaginary series in this case. Therefore $\m^{-1}$ can be pulled out of equation \eqref{eq:imaginarypart} to give: 
\begin{equation}
  \label{eq:im-linear}
\left.\left(\frac{\Delta_{\real}}{\Delta_{t, \real}}\right)\right|_{Q^+_{\imag}}=\m 
\end{equation}
This argument does {\em{not}} apply in indefinite Kac-Moody type. The restriction map is not generally linear over imaginary series because the sum of an imaginary vector with a real vector can be either real or imaginary. Formula \ref{eq:im-linear} is not true in the general Kac-Moody case.
 

\section{The Affine Case}\label{sect:affinecase}
In this section we give an explicit formula for $\m$ in the affine case. This formula is equivalent to Macdonald's Constant Term conjecture \cite{Macdonald82}, which was proven by Cherednik \cite{Cherednik95}--the equivalence is explained in \cite{Macdonald03}. The formula we give in Theorem \ref{thm:affineformula} can be deduced from formulas which have appeared in the literature before. Our contribution is to phrase it in a uniform way for all affine types, in particular including $A^{(2)}_{2r}$.  Such a concrete formula is probably unavailable in indefinite Kac-Moody type, but \S \ref{sect:ordertwovanishing} below generalizes some properties of this formula to all types.  

Consider an affine Kac-Moody root system $\Phi$. We specify a zeroth vertex in the Dynkin diagram as follows. Except in the case of $A^{(2)}_{2r}$, we choose $0$ to agree with \cite{Kac90}. In the case of $A^{(2)}_{2r}$, our vertex $0$ is Kac's vertex $r$. Note that these are the conventions of \cite{Beck-Nakajima04}, which are chosen to interact well with Drinfeld's loop presentation of the quantum affine algebra. The Dynkin diagram with vertex $0$ removed defines a finite root subsystem $\Phi_{\fin}$. Let $Q_\fin$ be the root lattice corresponding to this subsystem, and $Q^+_{\fin}$ its positive cone.

Recall that every imaginary root of $\fg$ is of the form $k \delta$, where $\delta$ is the minimal imaginary root. For each imaginary root $\eta$ we define $S(\eta) = \{ \beta \in  Q^+_{\fin} \suchthat \beta + \eta \in \Phi_{\real} \}$ and let:
\begin{align}
  \label{eq:6}
\m_\eta = \prod_{\beta \in S(\eta)} \frac{ (1 - t^{\height(\beta)} e^\eta)^2} { (1 - t^{\height(\beta)-1} e^\eta) (1 - t^{\height(\beta)+1} e^\eta) }  
\end{align}

Then we have the following theorem, which is equivalent to Cherednik's resolution of the constant term conjecture.
\begin{Theorem}\label{thm:affineformula}\cite{Cherednik95}, \cite[3.12]{Macdonald03}, \cite[5.8.20]{macbook}
  \begin{align}
    \label{eq:7}
    \m = \prod_{\eta  \in \Phi^+_{\imag}} \m_\eta
  \end{align}
\end{Theorem}

For every affine Kac-Moody root system other than $A^{(2)}_{2r}$, the identity \eqref{eq:7} follows from \cite[3.12]{Macdonald03}. For the root system $A^{(2)}_{2r}$ we may can derive the constant term formula \cite[5.8.20]{macbook} from that of the root system $(C_r^{\vee},C_r)$ following the specialization choices of Sharma and Viswanath \cite[Section 3]{sharma-viswanath}.

\subsection{Divisibility by $(1-t)^2$}
The results of \S \ref{sect:ordertwovanishing} below are generalizations of some aspects of Theorem \ref{thm:affineformula} to arbitrary Kac-Moody type. Specifically, Theorem \ref{thm:roots_only} will assert that $\m$ can always be expressed as a product over imaginary roots, as in \eqref{eq:7}.

The shape of the factors $\m_{\eta }$ in \eqref{eq:6} motivates Theorem \ref{thm:order2vanishing}, which asserts that a certain polynomial is divisible by $(1-t)^2.$ Write $\m_{\eta }$ in product form, as in \eqref{eq:2}: $\m_{\eta }=\prod_{n=0}^{N_{\eta}}(1-t^n e^{\eta})^{-m(\eta, n)}$. By \eqref{eq:6}, the polynomial
$m_{\eta}(t)=\sum_{n=0}^{N_{\eta}}m(\eta, n)t^{n}$ 
is divisible by $(1-t)^2$. We see that $m_{\eta}(t) = (1-t)^2 \chi_{\eta}(t)$, where: 
\begin{align}
  \label{eq:1}
  \chi_{\eta}(t) = \sum_{\beta \in S(\eta)} t^{\height(\beta)}
\end{align}

In indefinite type, it is no longer true that $\chi_{\eta}(t)$ is a polynomial with positive coefficients, but Conjecture \ref{conj:alternating} points to the possibility of some interesting positivity properties. 
\section{The Generalized Peterson Algorithm}\label{sect:peterson}
We now describe an efficient algorithm to compute $\m$ from \eqref{eq:imaginarypart}. This process relies on manipulations of power series via Propositions \ref{prop:relativeinverse} and \ref{prop:seriestoproduct}. It requires a list of positive real roots in $\Phi$ as input. The power series $\m$ in height up to $H$ is computed with runtime polynomial in $H$. The algorithm  generalizes the Peterson algorithm \cite{Peterson83} to compute root multiplicities, and in fact recovers the root multiplicities when specialized at $t=0$. It could be used for more general power series computations, but for concreteness we will describe its use to compute $\m$ only.

By Proposition \ref{prop:seriestoproduct}, $\m$ can be uniquely expressed as a product of the form:
\begin{equation}\label{eq:mprodform_Peterson}
\m=\prod_{\lambda \in Q^+_{\imag}} \prod_{n=0}^{N_{\lambda}}(1-t^n e^{\lambda})^{-m(\lambda, n)}
\end{equation}
with $N_\lambda \in \NN$ and $m(\lambda, n)\in \ZZ$. We have seen from \eqref{eq:denominatorformula} that $m(\lambda, 0)=\mult(\lambda )$ for $\lambda\in Q^+_{\imag}$. We extend the functions $m(\lambda, n)$ to all $\lambda\in Q^+$ by setting $N_{\alpha}=1,$ $m(\alpha, 0)=1$, $m(\alpha, 1)=-1$ for real roots $\alpha$, and $m(\lambda, n)=0$ for all $n$ if $\lambda$ is not an imaginary vector or a real root.
Then we can write:
\begin{equation}
\m^{-1}\frac{\Delta_{\real}}{\Delta_{t, \real}}=\prod_{\lambda \in Q^+} \prod_{n=0}^{N_{\lambda}}(1-t^n e^{\lambda})^{m(\lambda, n)}
\end{equation}
The restriction of this product to $Q^+_{\imag}$ is $1$ by \eqref{eq:imaginarypart}. We use this fact to compute its imaginary factors one at a time, inductively by the height of $\lambda$. 

For $\lambda \in Q^+$, let:
\begin{equation}\label{eq:Mlambdashape}
\m_{\lambda}=\prod_{n=0}^{N_{\lambda}} (1-t^n e^{\lambda})^{-m(\lambda, n)}
\end{equation}
To initialize the algorithm, we set $\m_0=1$. We must determine $\m_{\lambda}$ for $\lambda \in Q^+_{\imag}$, assuming that all $\m_{\mu}$ for $\text{ht}(\mu)<\text{ht}(\lambda)$ have already been computed. For the vectors $\mu \in Q^+_{\imag}$, $\m_{\mu}$ is computed inductively. For  $\mu \not \in Q^+_{\imag}$, $\m_{\mu}$ is known because a list of positive real roots of height less than $\text{ht}(\lambda)$ is known. 
By \eqref{eq:imaginarypart}, the expansion of the finite product:
\begin{equation} \label{eq:partialvanishingquotient}
\m_{\lambda}^{-1} \prod_{\substack{\mu \in Q^+ \\ \text{ht}(\mu)<\text{ht}(\lambda)}} \m_{\mu}^{-1}
\end{equation}
has coefficient $0$ at $\lambda$. The second factor here is assumed known by induction.

Because \eqref{eq:partialvanishingquotient} has $0$ as the coefficient of $e^\lambda$, the coefficient of $e^\lambda$ in the product $\displaystyle \prod_{\substack{\mu \in Q^+ \\ \text{ht}(\mu)<\text{ht}(\lambda)}} \m_{\mu}^{-1}$ is a polynomial in $t$ equal to $m(\lambda, 0) + m(\lambda, 1)t+\cdots+m(\lambda, N_{\lambda})t^{N_{\lambda}}$. Thus we have computed $\m_{\lambda}$ inductively. 

Note that each $\m_\lambda$ is a power series unit in $\Omega_1$, with constant coefficient $1$, and is regular at $t=0$. Any finite product of $\m_\lambda$ factors retains these properties. The infinite product $\m=\prod_{\lambda \in Q^+_{\imag}} \m_{\lambda}$ also has a well-defined power series expansion because only finitely many factors contribute to each coefficient. It is a power series unit in $\Omega_1$, with constant coefficient $1$, and is regular at $t=0$. It satisfies equation \eqref{eq:imaginarypart}, which defines $\m$ uniquely by Proposition \ref{prop:relativeinverse}.

The algorithm as described above produces a $W$-invariant series $\m.$ Since we know a priori that $\m$ is $W$-invariant we can exploit this fact to make the algorithm more efficient. If the imaginary vectors $\lambda$ and $\lambda'$ lie in the same $W$-orbit, then $m(\lambda', n)=m(\lambda, n)$ for all $n$. We need only to compute $m(\lambda, n)$ by the above algorithm for $\lambda$ of minimal height, i.e. in the antidominant subcone of $Q^+_{\imag}$. The remaining factors of $\m$ can be deduced from the fact that every $\lambda' \in Q^+_{\imag}$ is $W$-equivalent to an antidominant $\lambda.$

\section{The Generalized Berman-Moody Formula}\label{sect:GenBermanMoody}
We adapt the method of Berman and Moody \cite{Berman-Moody79} to give formulas for $m(\lambda, n)$. Our starting point is equation \eqref{eq:identity}. Set:
\begin{equation} \label{eq:sigma}
\Sigma = P(t)^{-1} \sum_{w\in W} t^{\ell(w)} \prod_{\alpha \in \Phi(w^{-1})} \frac{1-t^{-1}e^{\alpha}}{1-te^{\alpha}} = \frac{\Delta_{\real}}{\m \Delta_{t, \real}}
\end{equation}
The series coefficients on the left side are ratios of power series in $t$, but since the right-hand side lies in $\Omega_1$, they simplify to polynomials in $t$. Let $P_{\lambda}(t)$ denote the coefficient of $\lambda$ in $\Sigma.$ 

We may specify these coefficients using the expansion:
\begin{equation}
\frac{1-t^{-1}e^{\alpha}}{1-te^{\alpha}} = \left( 1 + (1-t^{-2}) \sum_{k=1}^{\infty} t^k e^{k \alpha} \right)
\end{equation}
For $\lambda \in Q^+$, let $K(\lambda)$ denote the set of real Kostant partitions of $\lambda$. That is, an element $\underline{k}\in K(\lambda)$ is a collection of nonnegative integers $k_{\alpha}$ indexed by $\alpha \in \Phi^+_{\real}$, all but finitely many of which are 0, with $\sum k_{\alpha} \alpha=\lambda$. Let $|\underline{k}|=\sum k_{\alpha}$, $\text{supp}(\underline{k})=\lbrace \alpha \in \Phi^+_{\real} | k_{\alpha} \neq 0 \rbrace$, and let $|\text{supp}(\underline{k})|$ denote the cardinality of $\text{supp}(\underline{k})$. A $w$-summand in $\Sigma$ contributes to the coefficient of $\lambda$ only if there exists a real Kostant partition $\underline{k}$ of $\lambda$ such that $\text{supp}(\underline{k}) \subset \Phi(w^{-1})$. In particular, $w^{-1}(\lambda)\in Q^-$. The contribution corresponding to a certain $w\in W$ and $\underline{k}\in K(\lambda)$ with $\text{supp}(\underline{k}) \subset \Phi(w^{-1})$ is $P(t)^{-1} t^{\ell(w)} (1-t^{-2})^{|\text{supp}(\underline{k})|} t^{|\underline{k}|}$. Therefore the full coefficient of $\lambda$ can be expressed as:
\begin{equation} \label{eq:plambda}
P_{\lambda}(t)= P(t)^{-1} \sum_{\underline{k} \in K(\lambda)} t^{|\underline{k}|} (1-t^{-2})^{|\text{supp}(\underline{k})|} P_{\text{supp}(\underline{k})}(t)
\end{equation}
where $P_{\text{supp}(\underline{k})}(t)$, an analogue of $P(t)$, is defined as:
\begin{equation}
P_{\text{supp}(\underline{k})}(t)=\sum_{\substack{w \in W \\ \text{supp}(\underline{k}) \subset \Phi(w^{-1})}} t^{\ell(w)}
\end{equation}

For example, when $\lambda\in Q^+_{\imag}$, \eqref{eq:plambda} implies that $P_{\lambda}(t)=0$, since there is no $w\in W$ with $w^{-1}(\lambda)<0$. When $\alpha_i$ is a simple root, it follows from \eqref{eq:plambda} that $P_{n\alpha_i}(t)=t^n(1-t^{-1})$. We do not see a direct way to argue that $P_{\lambda}(t)$ is a polynomial from equation \eqref{eq:plambda}, without using the generating function $\Sigma$. 


We introduce additional notation. For $\lambda\in Q^+$, let $\text{Par}(\lambda)$ be the set of vector partitions of $\lambda$. An element of $\text{Par}(\lambda)$ is an unordered tuple $\underline{\lambda}=(\lambda_i)$ of nonzero vectors in $Q^+$ such that $\lambda=\sum \lambda_i$. Let $|\underline{\lambda}|$ denote the number of components in the partition. We caution the reader that vector partitions and Kostant partitions both appear in the sequel, though we have endeavored to keep the two concepts separate. Roughly, our formula is indexed by vector partitions of $\lambda$, and then Kostant partitions of each component $\lambda_i$. 

We also introduce the set $\text{OPar}(\lambda)$ of ordered partitions of $\lambda$. If $\underline{\lambda} \in \text{Par}(\lambda)$ consists of $r_1$ copies of $\lambda_1$, $r_2$ copies of $\lambda_2$, etc., and $r$ components total, then the number of distinct orderings is the multinomial coefficient $\frac{r!}{r_1!\cdots r_k!}$, which we will abbreviate, following Berman and Moody, as $B(\underline{\lambda})$.

For $\kappa, \lambda \in Q^+$, if $\lambda=k \kappa$, we will write $\kappa| \lambda$ and $\frac{\lambda}{\kappa}=k$. Let $\mu$ denote the M\"{o}bius function. Recall that $m_{\lambda}(t)=\sum_{n=0}^{N_{\lambda}} m(\lambda, n) t^{n}.$ The main theorem of this section is the following:
\begin{Theorem} \label{thm:BermanMoody}
For all $\lambda \in Q^+$, we have the following:
\begin{equation} \label{eq:multformula}
m_{\lambda}(t) = \sum_{\kappa | \lambda} \mu\left(\lambda/\kappa\right)\left(\frac{\lambda}{\kappa}\right)^{-1} \sum_{\underline{\kappa}\in \mathrm{Par}(\kappa)} (-1)^{|\underline{\kappa}|}\frac{B(\underline{\kappa})}{|\underline{\kappa}|} \prod_{i=1}^{|\underline{\kappa}|} P_{\kappa_i}(t^{\lambda/\kappa})
\end{equation}
\end{Theorem}
This generalizes \cite[Theorem 2]{Berman-Moody79}, and recovers this theorem at $t=0$. Before commencing the proof, we state a version of \eqref{eq:multformula} where the sum is over ordered partitions. We have: 
\begin{equation} \label{eq:orderedmultformula}
m_{\lambda}(t) = \frac{1}{\text{ht}(\lambda)} \sum_{\kappa | \lambda} \mu(\lambda/\kappa) \sum_{\underline{\kappa}\in \text{OPar}(\kappa)} (-1)^{|\underline{\kappa}|} \text{ht}(\kappa_1) \prod_{i=1}^{|\underline{\kappa}|} P_{\kappa_i}(t^{\lambda/\kappa})
\end{equation}
The fact that formulas \eqref{eq:multformula} and \eqref{eq:orderedmultformula} are equivalent follows by counting possible orderings. If the unordered partition $\underline{\kappa}$ consists of $r_1$ copies of $\kappa_1$, $r_2$ copies of $\kappa_2$, etc., and $r$ components total, then the number of orderings with $\kappa_i$ first is $\frac{(r-1)!}{r_1!\cdots(r_i-1)!\cdots r_k!}=\frac{r_i}{r}B(\underline{\kappa})$. The sum of $\text{ht}(\kappa_1)$ over all orderings is $\frac{\sum r_i \text{ht}(\kappa_i)}{r}B(\underline{\kappa})= \frac{\text{ht}(\kappa)B(\underline{\kappa})}{|\underline{\kappa}|}$, which yields the equivalence. Therefore it suffices to prove the ordered version \eqref{eq:orderedmultformula}.

\begin{proof}
We have:
\begin{equation}
\Sigma= \prod_{\lambda\in Q^+} \prod_{n=0}^{N_{\lambda}} (1-t^n e^{\lambda})^{m(\lambda, n)}
\end{equation}
By taking the logarithm of both sides, and then applying the differential operator $D=\sum_i e^{\alpha_i} \frac{\partial}{\partial e^{\alpha_i}}$ we obtain:
\begin{equation}
\frac{D(\Sigma)}{\Sigma} = \sum_{\lambda\in Q^+} \sum_{n=0}^{N_{\lambda}} -m(\lambda, n) \text{ht}(\lambda) \sum_{k=1}^{\infty} t^{kn} e^{k \lambda}
\end{equation}

We now compute $\frac{D(\Sigma)}{\Sigma}$ in terms of the coefficients $P_{\lambda}(t)$. It is clear that: 
\begin{equation}
D(\Sigma)=\sum_{\lambda \in Q^+} -\text{ht}(\lambda) P_{\lambda}(t) e^{\lambda}
\end{equation}
Also, $\Sigma^{-1}$ can be expanded as a geometric series:
\begin{align}
\frac{1}{\Sigma}&=\frac{1}{1+\sum_{\lambda \in Q^+ \setminus \lbrace 0\rbrace} P_{\lambda}(t)e^{\lambda}} \\
&=\sum_{k=0}^{\infty} (-1)^k \left(\sum_{\lambda \in Q^+ \setminus \lbrace 0\rbrace} P_{\lambda}(t) e^{\lambda} \right)^k \\
&=\sum_{\lambda \in Q^+} \sum_{\underline{\lambda} \in \text{OPar}(\lambda)} (-1)^{|\underline{\lambda}|}\left( \prod_{i=1}^{|\underline{\lambda}|} P_{\lambda_i}(t) \right) e^{\lambda}
\end{align}
The term $\left(\sum_{\lambda \in Q^+ \setminus \lbrace 0\rbrace} P_{\lambda}(t) e^{\lambda} \right)^k$ expands as a sum of ordered partitions into $k$ components. We can expand the quotient  $\frac{D(\Sigma)}{\Sigma}$ in terms of ordered partitions as well, where each ordered partition is weighted by the height of its first entry:
\begin{equation}
\frac{D(\Sigma)}{\Sigma} = \sum_{\lambda \in Q^+} \sum_{\underline{\lambda} \in \text{OPar}(\lambda)} (-1)^{|\underline{\lambda}|} \text{ht}(\lambda_1) \left( \prod_{i=1}^{|\underline{\lambda}|} P_{\lambda_i}(t) \right) e^{\lambda}
\end{equation}
Note that the $\lambda=0$ term here is $0$. 

Comparing the coefficient of $e^{\lambda}$ in our two formulas for $\frac{D(\Sigma)}{\Sigma}$ we have:
\begin{equation}
\sum_{\kappa | \lambda} \sum_{n=0}^{N_{\kappa}} m(\kappa, n) \text{ht}(\kappa) t^{n \lambda/\kappa} = \sum_{\underline{\lambda} \in \text{OPar}(\lambda)} (-1)^{|\underline{\lambda}|} \text{ht}(\lambda_1) \left( \prod_{i=1}^{|\underline{\lambda}|} P_{\lambda_i}(t) \right)
\end{equation}
We would like to isolate the $\kappa=\lambda$ terms on the left side. A change of variables and M\"{o}bius inversion yields the expression:
\begin{equation}
\sum_{n=0}^{N_{\lambda}} m(\lambda, n) t^{n} = \frac{1}{\text{ht}(\lambda)} \sum_{\kappa | \lambda} \mu \left( \frac{\lambda}{\kappa}\right)  \sum_{\underline{\kappa} \in \text{OPar}(\kappa)} (-1)^{|\underline{\kappa}|} \text{ht}(\kappa_1) \left( \prod_{i=1}^{|\underline{\kappa}|} P_{\kappa_i}(t^{\lambda/\kappa}) \right)
\end{equation}
This is \eqref{eq:orderedmultformula}, and \eqref{eq:multformula} follows.
\end{proof}

By definition of $m(t, \lambda)$, equation \eqref{eq:multformula} must yield $1-t$ when $\lambda \in \Phi^+_{\real}$. Furthermore, when $\lambda \in Q^+ \setminus Q^+_{\imag}$ and $\lambda$ is not a root, equation \eqref{eq:multformula} must yield $0$. We do not see a direct argument for this result from the formula. The situation with $\lambda \in Q^+_{\imag}$ is of greatest interest. All $\kappa$ dividing $\lambda$ will belong to $Q^+_{\imag}$ as well. Notice, however, that the vector partitions of $\kappa$ which contribute nontrivially to the formula are only those with $\kappa_i \notin Q^+_{\imag}$ for all $i$, so that $P_{\kappa_i}(t)$ does not vanish. The sum is over partitions of imaginary vectors into real vectors. This will be of use in the proof of Theorem \ref{thm:order2vanishing}.


\subsection{Some explicit computations in rank $2$}\label{subsect:explicit_rank2_comp}

In this section we will compute some examples of the polynomials $P_{\text{supp}(\underline{k})}(t)$ in rank 2 root systems. These can be used to compute $P_{\lambda}(t)$ by \eqref{eq:plambda} and then $m_{\lambda}(t)$ by \eqref{eq:multformula}. In type $A_2$, for $i=1, 2$, we have:
\begin{align*}
& P(t)=1+2t+2t^2+t^3 \\
& P_{\{\alpha_i \}}(t)= t + t^2 + t^3 \\
& P_{\{\alpha_1+\alpha_2 \}}(t)= 2 t^2 + t^3 \\
& P_{\{\alpha_i, \alpha_1+\alpha_2 \}}(t)= t^2 + t^3 \\
& P_{\{\alpha_1, \alpha_2 \}}(t) = P_{\{\alpha_1, \alpha_2, \alpha_1+\alpha_2 \}}(t)= t^3
\end{align*}
For example, when $\lambda=\alpha_1+\alpha_2$, there are two Kostant partitions of $\lambda$. By \eqref{eq:plambda}, we can compute:
\begin{equation}
P_\lambda(t)= \frac{t^2(1-t^{-2})^2 t^3}{1+2t+2t^2+t^3} + \frac{t(1-t^{-2})(2t^2+t^3)}{1+2t+2t^2+t^3} = t^2-t
\end{equation}
Note that the sum is a polynomial in $t$, but the individual summands are not.

In all infinite rank 2 Kac-Moody root systems, the Weyl group can be presented as $W = \langle \s_1,\s_2 \suchthat \s_1^2 = \s_2^2 = 1 \rangle$, yielding:
\begin{align} \label{eq:8}
P(t) = \frac{1+t}{1-t} 
\end{align} 
The positive real roots can be partitioned into two sets:
\begin{equation}
\Phi^+_{\real} = \{ \alpha_1, \s_1\alpha_2, \s_1\s_2\alpha_1, \s_1\s_2\s_1\alpha_2, \ldots \} \sqcup \{ \alpha_2, \s_2\alpha_1, \s_2\s_1\alpha_2, \s_2\s_1\s_2\alpha_1, \ldots \}
\end{equation}
and each set $\Phi(w)$ for $w \in W$ is contained in one of the two components. Thus if some $\text{supp}(\underline{k})$ has non-empty intersection with both components, $P_{\text{supp}(\underline{k})} = 0$. If $\text{supp}(\underline{k})$ is contained in one component, suppose that the highest root in  $\text{supp}(\underline{k})$ is a product of $N-1$ simple reflections applied to a simple root. Then we have: 
\begin{align} \label{eq:13}
P_{\text{supp}(\underline{k})}(t) = \frac{t^{N}}{1-t}
\end{align}
In this case the individual summands of \eqref{eq:plambda} will be polynomials.

\section{Properties of $\m$}\label{sect:ordertwovanishing}
In this section we prove certain fundamental properties of $\m$, generalizing statements of \S \ref{sect:affinecase} to arbitrary Kac-Moody type. We begin with two propositions describing the behavior of $\m$ with respect to particular root subsystems.

\begin{Proposition} \label{prop:sublattice_restriction}
Let $\Phi_1$ be a root subsystem of $\Phi$ corresponding to a subset of vertices in the Dynkin diagram. Let $Q_1\subset Q$ be the corresponding root sublattice. Let $\m_1$ and $\m$ be defined as in \eqref{eq:mdefinition} for the root systems $\Phi_1$ and $\Phi$ respectively. Then:
\begin{equation}
\m|_{Q_1}=\m_1
\end{equation}
\end{Proposition}

\begin{proof}
We will use the characterization of $\m$ and $\m_1$ from \eqref{eq:imaginarypart}. Since $\m^{-1}$ is the inverse of $\frac{\Delta_{\real}}{\Delta_{t, \real}}$ relative to the cone $Q^+_{\imag}$ and since restriction to the sublattice $Q_1$ is a ring homomorphism, we have that $(\m|_{Q_1})^{-1}$ is the inverse of $\frac{\Delta_{\real}}{\Delta_{t, \real}}|_{Q_1}$ with respect to the cone $Q^+_{\imag}\cap Q_1$. Since the roots of $\Phi_1$ are precisely those roots of $\Phi$ which lie in $Q_1$, $\frac{\Delta_{\real}}{\Delta_{t, \real}}|_{Q_1}$ is equal to the corresponding ratio for the root system $\Phi_1$. Thus it suffices to prove that  $Q^+_{\imag}\cap Q_1=Q^+_{1, \imag}$.

Recall that $Q^+_{\imag}$ is the set of vectors in $Q^+$ all of whose Weyl group translates also lie in $Q^+$. Since the Weyl group of $\Phi_1$ is a subgroup of the Weyl group of $\Phi$, a vector in $Q^+_{\imag}\cap Q_1$ must lie in $Q^+_{1, \imag}$. To prove the opposite inclusion, we use \cite[Theorem 5.4]{Kac90}. A positive imaginary vector of $Q_1$ is conjugate under the Weyl group of $\Phi_1$ to a positive antidominant vector in $Q_1$, which is either an imaginary root or a sum of imaginary roots. This vector is still positive and antidominant in $Q$, so it is still either an imaginary root or a sum of imaginary roots in $Q$. Thus $Q^+_{1, \imag}\subset Q^+_{\imag}$.
\end{proof} 

\begin{Proposition}\label{prop:reducible_case}
Let $\Phi$ be a reducible root system and let $\Phi_1$, $\Phi_2$ be root subsystems corresponding to a partition of the simple roots of $\Phi$ into mutually orthogonal subsets. Let $\m_1$, $\m_2$, and $\m$ be defined as in  \eqref{eq:mdefinition} for the root systems $\Phi_1$, $\Phi_2$, and $\Phi$ respectively. Then:
\begin{equation}
\m=\m_1\m_2
\end{equation}
\end{Proposition}

\begin{proof}
In this case, the Weyl group of $\Phi$ decomposes as a direct product, and therefore the imaginary cone $Q^+_{\imag}=Q^+_{1, \imag} \oplus Q^+_{2, \imag}$. Further, each real root of $\Phi$ is either a real root of $\Phi_1$ or $\Phi_2$, so $\frac{\Delta_{\real}}{\Delta_{t, \real}}$ can be factored into a product over $\Phi_{1, \real}$ and a product over $\Phi_{2, \real}$. Since $\m_1$ is the inverse of the first product relative to the cone $Q^+_{1, \imag}$, and $\m_2$ is the inverse of the second product relative to the cone $Q^+_{2, \imag}$, $\m_1\m_2$ is the inverse of $\frac{\Delta_{\real}}{\Delta_{t, \real}}$ relative to $Q^+_{\imag}$. 
\end{proof}

With these propositions in hand it is possible to prove that $\m$ can always be expressed as a product over positive imaginary roots, as in Theorem \ref{thm:affineformula}.

\begin{Theorem} \label{thm:roots_only}
For $\lambda \in Q^+$ not a root, and $n\in \ZZ$, we have $m(\lambda, n)=0$. 
\end{Theorem}
\begin{proof}
For $\lambda \notin Q^+_{\imag}$, this is clear from the construction of $\m$. For $\lambda \in Q^+_{\imag}$ not a root, by \cite[Theorem 5.4]{Kac90}, $\lambda$ is $W$-conjugate to some antidominant vector whose support is disconnected in the Dynkin diagram of $\Phi$. Thus it suffices to consider prove the statement when $\lambda$ has disconnected support. By Proposition \ref{prop:sublattice_restriction}, $m(\lambda, n)$ does not change if we restrict to the root subsystem corresponding to $\supp(\lambda)$. By Proposition \ref{prop:reducible_case}, the series $\m$ on this root subsystem is a product with no factors at the vector $\lambda$. Thus, $m(\lambda, n)=0$.
\end{proof}  

The following theorem generalizes \eqref{eq:1} to assert that the generating polynomial $m_{\lambda}(t) = \sum_{n=0}^{N_{\lambda}} m(\lambda, n) t^n$ is divisible by $(1-t)^2$.

\begin{Theorem}\label{thm:order2vanishing}
For all nonzero $\lambda \in Q^+$, $m_{\lambda}(t)$ has a root at $t=1$. For nonzero $\lambda \in Q^+_{\imag}$, this is a double root.
\end{Theorem}
\begin{proof}
We first show that for all nonzero $\lambda \in Q^+$, the polynomial $P_{\lambda}(t)$ has a root at $t=1$. If we evaluate $\Sigma$ at $t=1$, we have:
\begin{equation}
\m|_{t=1} = \sum_{\lambda \in Q^+} P_{\lambda}(1) e^{\lambda}
\end{equation}
The left side is supported on $ Q^+_{\imag}$, but on the right we have $P_{\lambda}(t)=0$ for all nonzero $\lambda \in Q^+_{\imag}$. Therefore $P_{\lambda}(1)=0$ for all nonzero $\lambda \in Q^+$.

Next, we apply formula \eqref{eq:multformula}:
\begin{equation}
m_{\lambda}(t) = \sum_{\kappa | \lambda} \mu\left(\lambda/\kappa\right)\left(\frac{\lambda}{\kappa}\right)^{-1} \sum_{\underline{\kappa}\in \text{Par}(\kappa)} (-1)^{|\underline{\kappa}|}\frac{B(\underline{\kappa})}{|\underline{\kappa}|} \prod_{i=1}^{|\underline{\kappa}|} P_{\kappa_i}(t^{\lambda/\kappa})
\end{equation}
As we have shown, each $P_{\kappa_i}(t^{\lambda/\kappa})$ will have a root at $t=1$, so the full expression has a root at $t=1$. Moreover, if $\lambda$ is an imaginary vector, then each $\kappa$ dividing $\lambda$ is imaginary as well. The only vector partitions of $\kappa$ which contribute to the formula are partitions into vectors outside the cone $Q^+_{\imag}$. In particular, such partitions must have at least two components. Therefore each product $\prod_{i=1}^{|\underline{\kappa}|} P_{\kappa_i}(t^{\lambda/\kappa})$ must have at least a double root at $t=1$, so the full expression has a double root.
\end{proof}

Therefore, for each imaginary root $\lambda$, we can define $\chi_\lambda = \frac{\m_\lambda(t)}{(1-t)^2}$, which by Theorem \ref{thm:order2vanishing} is an element of $\ZZ[t]$. We believe that $\chi_\lambda$ is the more fundamental object, as in the affine case. Some examples of this polynomial are tabulated in Appendix \ref{sect:dataandconjectures}.


\bibliographystyle{amsalpha}
\bibliography{papers}

\appendix

\section{Examples of the polynomial $\chi_{\lambda }$}\label{sect:dataandconjectures}
At the end of \S \ref{sect:ordertwovanishing} we set $\chi_\lambda(t) = \frac{m_\lambda(t)}{(1-t)^2}$ for imaginary roots $\lambda$. Theorem \ref{thm:order2vanishing} tells us that $\chi_\lambda \in \ZZ[t]$. In this section we tabulate certain polynomials $\chi_\lambda(t)$, computed with the algorithm of \S \ref{sect:peterson}. 

\noindent For the rank 2 hyperbolic root system with Cartan matrix $\twomat{2}{-3}{-2}{2}$, we have:
\begin{table}[h!]
\centering
\begin{tabular}{cp{35pc}}
$\lambda$ & $\chi_{\lambda}$ \\
$\left(1, 1\right)$ & $1$ \\
$\left(2, 2\right)$ & $-t + 1$ \\
$\left(3, 2\right)$ & $t^{2} + 2$ \\
$\left(3, 3\right)$ & $-t^{3} - 2 t + 2$ \\
$\left(4, 3\right)$ & $t^{4} - t^{3} + 2 t^{2} - 3 t + 3$ \\
$\left(4, 4\right)$ & $-t^{5} + t^{4} - 2 t^{3} + 3 t^{2} - 6 t + 3$ \\
$\left(5, 4\right)$ & $t^{6} - 2 t^{5} + 4 t^{4} - 6 t^{3} + 9 t^{2} - 9 t + 6$ \\
$\left(5, 5\right)$ & $-t^{7} + t^{6} - 4 t^{5} + 6 t^{4} - 10 t^{3} + 13 t^{2} - 13 t + 7$ \\
$\left(6, 4\right)$ & $t^{6} - 4 t^{5} + 5 t^{4} - 8 t^{3} + 11 t^{2} - 13 t + 6$ \\
$\cdots$ & $\cdots$ \\
$\left(10, 9\right)$ & $t^{16} - 7 t^{15} + 29 t^{14} - 91 t^{13} + 248 t^{12} - 584 t^{11} + 1197 t^{10} - 2170 t^{9} + 3505 t^{8} - 5039 t^{7} + 6437 t^{6} - 7253 t^{5} + 7042 t^{4} - 5618 t^{3} + 3405 t^{2} - 1372 t + 272$ 
\end{tabular}
\label{table:Cartan2322}
\end{table}

\noindent For the symmetric rank 2 hyperbolic root system with Cartan matrix $\twomat{2}{-3}{-3}{2}$:
\begin{table}[h!]
\begin{tabular}{cp{35pc}}
$\lambda$ & $\chi_{\lambda}$ \\
$\left(1, 1\right)$ & $1$ \\
$\left(2, 2\right)$ & $-2 t + 1$ \\
$\left(2, 3\right)$ & $t^{2} - t + 2$ \\
$\left(3, 2\right)$ & $t^{2} - t + 2$ \\
$\left(3, 3\right)$ & $-2 t^{3} + 3 t^{2} - 4 t + 3$ \\
$\left(3, 4\right)$ & $t^{4} - 3 t^{3} + 6 t^{2} - 6 t + 4$ \\
$\left(4, 3\right)$ & $t^{4} - 3 t^{3} + 6 t^{2} - 6 t + 4$ \\
$\left(4, 4\right)$ & $-2 t^{5} + 7 t^{4} - 12 t^{3} + 17 t^{2} - 16 t + 6$ \\
$\left(4, 5\right)$ & $t^{6} - 5 t^{5} + 15 t^{4} - 26 t^{3} + 30 t^{2} - 23 t + 9$ \\
$\left(4, 6\right)$ & $t^{6} - 8 t^{5} + 19 t^{4} - 31 t^{3} + 36 t^{2} - 28 t + 9$ \\
$\left(5, 4\right)$ & $t^{6} - 5 t^{5} + 15 t^{4} - 26 t^{3} + 30 t^{2} - 23 t + 9$ \\
$\left(5, 5\right)$ & $-2 t^{7} + 9 t^{6} - 30 t^{5} + 58 t^{4} - 82 t^{3} + 77 t^{2} - 50 t + 16$ \\
$\left(6, 4\right)$ & $t^{6} - 8 t^{5} + 19 t^{4} - 31 t^{3} + 36 t^{2} - 28 t + 9$ \\
$\cdots$ & $\cdots$ \\
$\left(10, 9\right)$ & $t^{16} - 15 t^{15} + 135 t^{14} - 811 t^{13} + 3535 t^{12} - 11729 t^{11} + 30615 t^{10} - 64282 t^{9} + 110096 t^{8} - 154852 t^{7} + 178868 t^{6} - 168420 t^{5} + 127110 t^{4} - 74539 t^{3} + 32094 t^{2} - 9070 t + 1267$ \\
\end{tabular}
\label{table:Cartan2332}
\end{table}

\noindent For the Feingold-Frenkel rank 3 hyperbolic root system with Cartan matrix $\threemat{2}{-2}{0}{-2}{2}{-1}{0}{-1}{2}$:
\begin{table}[h!]
\begin{tabular}{cp{35pc}}
$\lambda$ & $\chi_{\lambda}$ \\
$\left(1, 1, 0\right)$ & $1$ \\
$\left(2, 2, 0\right)$ & $1$ \\
$\left(2, 2, 1\right)$ & $2$ \\
$\left(3, 3, 0\right)$ & $1$ \\
$\left(3, 3, 1\right)$ & $-t + 3$ \\
$\left(3, 4, 2\right)$ & $-2 t + 5$ \\
$\left(4, 4, 0\right)$ & $1$ \\
$\left(4, 4, 1\right)$ & $-2 t + 5$ \\
$\left(4, 4, 2\right)$ & $-t^{2} - 6 t + 7$ \\
$\left(4, 5, 2\right)$ & $t^{3} + t^{2} - 9 t + 11$ \\
$\left(5, 5, 0\right)$ & $1$ \\
$\left(5, 5, 1\right)$ & $-5 t + 7$ \\
$\left(5, 5, 2\right)$ & $2 t^{3} + 2 t^{2} - 17 t + 15$ \\
$\left(5, 6, 2\right)$ & $-t^{4} + 3 t^{3} + 6 t^{2} - 26 t + 22$ \\
$\left(5, 6, 3\right)$ & $-3 t^{4} + 6 t^{3} + 13 t^{2} - 43 t + 30$ \\
$\left(6, 6, 0\right)$ & $1$ \\
$\left(6, 6, 1\right)$ & $t^{2} - 8 t + 11$ \\
$\left(6, 6, 2\right)$ & $-2 t^{4} + 5 t^{3} + 11 t^{2} - 43 t + 30$ \\
$\left(6, 6, 3\right)$ & $-6 t^{4} + 8 t^{3} + 23 t^{2} - 65 t + 42$ \\
$\left(6, 7, 2\right)$ & $-5 t^{4} + 6 t^{3} + 22 t^{2} - 63 t + 42$ \\
$\left(7, 7, 0\right)$ & $1$ \\
$\left(7, 7, 1\right)$ & $2 t^{2} - 15 t + 15$ \\
\end{tabular}
\label{table:Cartan220221012}
\end{table}

\noindent Notice that the phenomenon of $\chi_\lambda = 1$ for $\lambda=(n, n, 0)$ is an instance of Proposition \ref{prop:sublattice_restriction} and Theorem \ref{thm:affineformula}, as these are roots of the embedded affine root subsystem $A_1^{(1)}$.



\end{document}